\documentclass[a4paper,10pt]{article}
\usepackage{amsmath}
\usepackage{amsfonts}
\usepackage{a4wide}
\usepackage{amssymb}
\usepackage{amsthm}
\theoremstyle{plain}
\begingroup
\newtheorem{theorem}{Theorem}
\newtheorem*{mtheorem}{Main Theorem}
\newtheorem{lemma}[theorem]{Lemma}

\newtheorem{definition}[theorem]{Definition}
\newtheorem{rmk}[theorem]{Remark}

\newtheorem{question}[theorem]{Question}
\endgroup

\theoremstyle{remark}
\begingroup
\endgroup

\mathchardef\emptyset="001F

\numberwithin{equation}{section}

\newcommand{\op}[1]{{\rm{#1}}}

\title{Decorrelation as an avatar of convexity}
\author{Mircea Petrache\footnote{UPMC Univ. Paris 6,  UMR 7598 Laboratoire Jacques-Louis Lions,  Paris, F-75005 France}
}
\date{July 2, 2015}
\begin{document}
 \maketitle
 \begin{abstract}
 If $X$ is a Polish space then we show that the product measure on $X^\infty$ is guaranteed to minimize $c$-energy amongst exchangeable measures with fixed marginals if and only if the interaction kernel $c$ defines a convex energy functional on probability measures. A reformulation of this condition close to the theory of positive definite functions is highlighted.
\end{abstract}

\section{Introduction}
This paper concerns a decorrelation phenomenon described in \cite{CFP}, within the study of mean field limits in quantum physics models motivated by \cite{lieb1} and more precisely taking place within the Kohn-Sham formulation \cite{kohnsham}. A question of interest is to rigorously explain the reason why decorrelation occurs. This question is relevant in the more general framework of optimal transportation with infinitely many marginals and general costs \cite{pass}. For a broader overview see also the recent survey \cite{dgn} and the references therein. What we found is that a rather simple convexity condition turns out to be equivalent to decorrelation, and this link is extremely robust and close to the classical topic of positive definite functions.

\subsection{Setting and main question}
Let $X$ be a separable metric space, endowed with a function on pairs of points $c:X\times X\to \mathbb R^+\cup \{+\infty\}$. Fix a Borel probability measure, $\mu\in\mathcal P(X)$ such that 
\begin{equation}\label{mufinite}
 \int_{X^2}c(x,y)\mu(dx)\mu(dy)<\infty.
\end{equation}
Sometimes $\mu(A)$ is interpreted as the probability of finding a particle in the region $A$ and $c$ represents some interaction energy between pairs of these particles.\\

We then consider the case of countably many copies of $\mu$ interacting with each other, which is modeled as follows. We endow $X^\infty$ with the product Borel $\sigma$-algebra and consider the so-called \emph{exchangeable} probability measures 
\[\mathcal P_{sym}(X^\infty)\]
on it \cite{diaconisfreedman}. An exchangeable measure is a measure which is invariant under all of those permutations of the copies of $X$ which leave all but finitely many of the copies invariant. For $\gamma\in\mathcal P_{sym}(X^\infty)$ we write
\[
 \gamma\mapsto\mu
\]
if the image measure of $\gamma$ under projection to the first (thus by invariance, under any) coordinate, equals $\mu$. Following the interpretation from above, such $\gamma$ may represent a probability distribution of states of a system consisting of infinitely many indistinguishable copies of a given particle. We say that such $\gamma$ is \emph{decorrelated} if it equals
\begin{equation}\label{decorrel}
 \gamma_0:=\mu^\infty.
\end{equation}
The natural way to define the $c$-interaction energy for such symmetric $\gamma$ consists of taking the following $\gamma$-average of pair interactions:
\begin{equation}\label{cgamma}
 \langle c,\gamma\rangle:=\lim_{N\to\infty}\frac{2}{N(N-1)}\sum_{1\le i <j\le N}\int_{X^\infty}c(x_i,x_j)d\gamma(x_1,x_2,\ldots).
\end{equation}
The question that we want to investigate is the following:

\begin{question}\label{q}
 Under which conditions on $c$ is $\gamma_0$ like in \eqref{decorrel} a minimizer (resp. unique minimizer) of $\langle c,\gamma\rangle$ among $\gamma\in\mathcal P_{sym}(X^\infty)$ such that $\gamma\mapsto\mu$?
\end{question}

 Our answer Question \ref{q} is to single out a \emph{necessary and sufficient} condition on $c$ for decorrelation valid in the same generality as that in which the problem makes sense. The condition is a convexity requirement on the energy operator associated to $c$. This can be very concretely reformulated as a \emph{balanced positive definiteness} condition on $c$, which gives a direct link to the classical theory of positive definite functions (see Definition \ref{posdef} and Lemma \ref{defposdef} below). We emphasize that costs $c$ like the examples in Theorem 1.2 of \cite{CFP}, but also less regular ones, are included in our setting. See Section \ref{posdefsec} for more examples.

\begin{mtheorem}\label{decorthm}
Let $X$ be a Polish space and consider a lower semi-continuous symmetric cost $c:X\times X\to [0,\infty]$.
\begin{enumerate}
 \item Assume that the functional
 \begin{equation}\label{k}
    Q\mapsto K(Q):=\int_{X\times X}c(x,y)dQ(x)dQ(y)\in [0,\infty] 
 \end{equation}
is convex on $\mathcal P(X)$. Suppose that $\mu\in\mathcal P(X)$ is such that $K(\mu)<\infty$. Then the measure \eqref{decorrel} is a minimizer of the infinite-body optimal transport problem 
\begin{equation}\label{min2}
 \inf\left\{\langle c,\gamma\rangle:\ \gamma\in\mathcal P_{sym}(X^\infty),\ \gamma\mapsto\mu\right\}.
\end{equation}
\item Assume that for all $\mu\in\mathcal P(X)\cap\{K<\infty\}$ the measure $\gamma_0=\mu^{\otimes\infty}$ is a minimizer of \eqref{min2}. Then $K$ is convex on $\mathcal P(X)$.
\end{enumerate}
Moreover the functional $K$ as in \eqref{k} is strictly convex on $\mathcal P(X)\cap\{K<\infty\}$ if and only if for any $\mu$ in this last set $\gamma_0$ is the unique minimizer of \eqref{min2}.
\end{mtheorem}
The discussion of the sharpness of our hypotheses is as follows.
\begin{itemize}
 \item The hypothesis that $X$ is Polish: this can be replaced (in the statement and also in our proof) by the weaker assumption that $X$ is a \emph{locally compact separable space}, see Remark \ref{extend}. However these seem to be the minimal topological assumptions needed for applying the most general De Finetti-Hewitt-Savage theorem \cite{hs} Theorem 7.4, as counterexamples to the representation theorems are known in more nonstandard measure spaces, as explained at the end of Section \ref{defisec}. 
 \item The regularity hypotheses on $c$: lower semi-continuity is essential to ensure the existence of a minimizer in \eqref{min2}. The symmetry requirement is harmless, since for $\gamma$ as in \eqref{min2} the value of $\langle c,\gamma\rangle$ is equal on $c$ and on its symmetrization. 
 \item The sign hypothesis on $c$: the minimization problem \eqref{min2} is insensitive to changing the cost $c$ by an additive constant, so our result trivially extends to the case where $c$ is required to be bounded below rather than nonnegative. If $c$ is unbounded below on $(\op{spt}\mu)^2$ then either \eqref{min2} loses interest and the infimum is $-\infty$, or $c$ is also unbounded above and then the limit in formula \eqref{cgamma} may not exist.
\end{itemize}
 An interesting possible direction for future developments is to look at $c$ unbounded both above and below and adapted classes of measures $\mu$. To generalize the present sharp characterization in that case, one must then base the whole discussion on more specialized (and to be defined and motivated) regularity classes of $c,\gamma$ that allow to give a meaning to \eqref{cgamma}. We leave this endeavor to future work.

\section{The De Finetti-Hewitt-Savage theorem}\label{defisec}
We let $X$ be a Polish space in this section, unless explicitly stated otherwise. When we speak of $\mathcal P(X)$ we implicitly select the $\sigma$-algebra $\mathcal B$ of Borel sets on $X$. In metrizable spaces this algebra is also the smallest one making continuous functions measurable \cite{boga}. Measures in $\mathcal P(X^\infty)$ are assumed measurable with respect to the $\sigma$-algebra $\mathcal B^\infty$ generated by the cylindrical sets $A_1\times\cdots\times A_k\times X^\infty$ with $k\in\mathbb N, A_i\in \mathcal B$.\\

When we speak of $\mathcal P(\mathcal P(X))$ the $\sigma$-algebra $\mathcal B^*$ that we consider on $\mathcal P(X)$ is the one of the Borel sets generated by open sets for the weak-* topology. This is the smallest topology making the functions $\mathcal P(X)\ni Q\mapsto Q(A)$ continuous for all $A\in \mathcal B$, so $\mathcal B^*$ is the smallest $\sigma$-algebra making all such functions measurable. It also makes all maps $\mathcal P(X)\ni Q\mapsto \int f dQ$ measurable for all continuous and bounded $f:X\to \mathbb R$.\\

For $Q\in\mathcal P(X)$ we define $Q^{\otimes\infty}\in \mathcal P_{sym}(X^\infty)$ by requiring 
\[Q^{\otimes\infty}(A_1\times\cdots\times A_k\times X^\infty):=\prod_{i=1}^kQ(A_i)\ \text{ for }A_i\in\mathcal B, i=1,\ldots,k.\]
With these assumptions, we mention the following result, also called De Finetti-Hewitt-Savage theorem, which we reformulate for a Polish space:
\begin{theorem}[Hewitt-Savage \cite{hs} Thm. 7.4, \cite{diaconisfreedman} Thm. 20]\label{hsthm}
 Let $X$ be a Polish space and $P\in\mathcal P_{sym}(X^\infty)$. Then there exists a unique measure $\mu\in\mathcal P(\mathcal P(X))$ such that 
 \[
 P=\int_{\mathcal P(X)} Q^{\otimes \infty} d\mu(Q).
 \]
\end{theorem}
Note that the above result can be reformulated in the measurable space category, i.e. without reference to the metric or topology of $X$. In \cite{hs} the theorem is formulated more abstractly and in a slightly wider generality, namely for the Baire $\sigma$-algebra of a locally compact separable space. The proof of the extension is via the $1$-point compactification of the underlying space $X$, compactification which is then Polish.\\
A simpler proof of Theorem \ref{hsthm} is present in \cite{diaconisfreedman} Thm. 14 and again it uses a topology on $X$. As noted after the proof of that theorem (\cite{diaconisfreedman} p.751), the proof in that case relies on the approximation of general $\mu\in\mathcal P(\mathcal P(X))$ by atomic measures in the weak-* topology, which itself can only hold if $\mathcal P(X)$ (and thus $X$ itself) is Polish. The extension from \cite{hs} Thm.7.2 to \cite{hs} Thm.7.4 then applies also to \cite{diaconisfreedman}.\\
A third more abstract proof is due to Ressel \cite{ressel}, who reduces the above result to a similar in spirit extreme-point representation theorem for the class of exponentially bounded positive definite functions on general semigroups (see \cite{semi} for the terminology). This formalizes in a neat way some of the ideas of Hewitt-Savage's initial proof, via the concept of a semigroup of separating functions for a $\sigma$-algebra.\\

Overall these available methods seem to not extend much beyond the case of Baire $\sigma$-algebras of separable locally compact spaces and in fact the above result seems to be quite sharp, as pointed out by some negative results. Notably, in \cite{dubinsfreedman}, Thm. 2.14, an example of measure space $(X,\mathcal F)$ (obtained by restricting the Borel sets of $[0,1]$ to a ``very non-measurable'' $X\subset [0,1]$) and $P\in\mathcal P_{sym}(X^\infty,\mathcal F^\infty)$ are constructed, such that the thesis of Theorem \ref{hsthm} fails. Thus the statement of the representation theorem is not robust enough to resist emancipation from the setting of a Borel $\sigma$-algebra of a Polish space tout court. This does not rule out the possibility that weaker but more complex conditions on $\sigma$-algebras might extend the above result.

\section{Positive definite functions}\label{posdefsec}
We give first a condition in the setting of more regular $c$, which is formulated more elementarily.

\begin{definition}\label{posdef}
 Let $X$ be a set and consider a continuous bounded function $c:X\times X\to[0,\infty[$ such that $c(x,y)=c(y,x)$ for all $x,y\in X$. The function $c$ is then said to be \emph{balanced positive definite} in case the following is true for arbitrary distinct points $x_1,\ldots,x_n\in X$ and arbitrary real numbers $a_1,\ldots,a_n$:
 \[
  \sum_{i,j=1}^nc(x_i,x_j)a_ia_j\ge 0 \text{ whenever } \sum_{i=1}^n a_i=0.
 \]
 \end{definition}

The more classical notion (of which our condition is a weakening) is \emph{positive definiteness}, where one requires the positivity condition on all real sequences, not just on the balanced ones. In applications the stronger condition of positive definiteness seems to be easier to check.\\ 

In a Polish space and for bounded $c$ the balanced positive definiteness and the positive definiteness become equivalent to the ones in terms of $K$ on $\mathcal P(X)$ or on finite Borel measures $\mathcal M(X)$, as formulated in our Main Theorem:

\begin{lemma}\label{defposdef}
 Let $X$ be a Polish space and $c:X\times X\to[0,\infty[$ a continuous bounded function. For $Q\in\mathcal M(X)$ define $K(Q)$ like in \eqref{k}. Then the following hold
 \begin{itemize}
  \item $c$ is positive definite if and only if $K$ is convex on $\mathcal M(X)$.
  \item $c$ is balanced positive definite if and only if $K$ is convex on $\mathcal P(X)$.
 \end{itemize}
\end{lemma}

\begin{proof}
By homogeneity of $K$ we may reduce to the case of measures of total mass $\le1$ and to $a_i$ such that $\sum_i a_i=1$. The implications from right to left follow by testing $K$ on atomic measures. They hold under a measurability hypothesis only on $c$. For the opposite implications: the hypotheses on $X$ imply that atomic measures are weakly dense in $\mathcal M(X)$ and balanced atomic probability measures are weakly dense in the cone over $\mathcal P(X) - \mathcal P(X)$. By taking weak limits we see that $K\ge 0$ on these spaces. The following formula can be applied either to $Q,Q'$ in $\mathcal M_1(X)$ or in $\mathcal P(X)$ and implies the theses: 
\[
 0\le K\left(\frac{Q-Q'}{2}\right)=\frac12\left(K(Q) + K(Q')\right) - K\left(\frac{Q+Q'}{2}\right).
\]
\end{proof}

\begin{rmk}
We don't use balanced positive definiteness of $c$ directly, and as a consequence of this the main theorem deals with more general $c$. However Definition \ref{posdef} seems closer to the terminology used in the literature.
\end{rmk}

We list below some results from which applications to our result follow, either via Lemma \ref{defposdef} or otherwise. The aim is rather to give hints to the diversity of possible applications rather than to try and exhaust them. 
\begin{itemize}
 \item The $c(x,y)=\ell(x-y)$ as in the first case of Theorem 1.2 of \cite{CFP}, i.e. the ones with $\ell\in L^1(\mathbb R^d)\cap C_b(\mathbb R^d), \hat\ell\ge 0$ are positive definite, strictly so if $\hat\ell>0$. 
 \item Any power law energy $c(x,y)=|x-y|^{-s}$ on $\mathbb R^d$ for $0<s<d$ satisfies the hypotheses of the Main Theorem, and by the Fourier transform criterion of \cite{CFP} gives a strictly convex $K$ on $\mathcal M(\mathbb R^d)\cap \{K<\infty\}$. The same holds for $c(x,y)=|\log|x-y||$ in $X=\mathbb R^d, d\ge 1$.
 \item An remarkable property of positive definite functions is that they are \emph{closed under pointwise products}. This follows from the definition via the similar property for finite matrices. The same holds also for smaller classes where we impose further linear constraints on the coefficients, i.e. for balanced positive definite functions.
 \item Some characterizations of positive definite functions for the case where $X$ is $\mathbb R^n$ (with possibly $n=\infty$), e.g. like Berstein's theorem and its analogues, are present in \cite{Schoenberg}. Any one of the functions as in \cite{Schoenberg} thus verifies the condition of our Main Theorem.
 \item In the case of $X=\mathbb S^d$ a well-known \cite{schoenberg2} class of positive definite functions are characterizable among those of the form $c(x,y)=\ell(\langle x,y\rangle)$, where the scalar product is taken in $\mathbb R^{d+1}$. If $P^\lambda_n$ are the ultraspherical polynomials defined e.g. by the Rodrigues formula \cite{muller}
 \[
  P^\lambda_n(t)=\frac{(-2)^n\Gamma(n+\lambda)\Gamma(n+2\lambda)}{n!\Gamma(\lambda)\Gamma(2n+2\lambda)}(1-t^2)^{1/2+\lambda}\left(\frac{d}{dt}\right)^n(1-t^2)^{n+\lambda - 1/2},
 \]
then $c$ as above is positive (resp. strictly positive) definite if and only if $\ell$ is expressible as 
 \[
  \ell(t)=\sum_{n=0}^\infty a_n P^{(d-1)/2}_n(t),
 \]
 with positive (resp. strictly positive) coefficients $a_n$. 
 \item The above example can be seen (using the product formula for spherical harmonics \cite{muller} Thm. 6) as a special case of the following general framework. Assume that $H(X)$ is a Hilbert space of functions on $X$ which separate points, and that $\{\psi_n\}_{n\in \mathbb N}$ is an orthonormal basis of functions in $H(X)$. Then $c(x,y)$ is positive (resp. strictly positive) definite if we can express it as
 \[
  c(x,y)=\sum_{n=1}^\infty a_n\psi_n(x)\psi_n(y)
 \]
with positive (resp. strictly positive) coefficients $a_n$.
\end{itemize}

\section{Proof of the main result and extension}

\begin{proof}[Proof of the Main Theorem:]
By Theorem \ref{hsthm} any $\gamma\in\mathcal P_{sym}(X^\infty)$ can be represented as 
 \[
  \gamma=\int_{\mathcal P(X)}Q^{\otimes\infty}d\nu(Q)
 \]
for some probability measure $\nu$ on $\mathcal P(X)$. Then the averaging in formula \eqref{cgamma} gives 
\begin{eqnarray*}
 \frac{2}{N(N-1)}\int_{X^\infty}\sum_{1\le i <j\le N}c(x_i,x_j)d\gamma&=&\int_{\mathcal P(X)}\left(\frac{2}{N(N-1)}\int_{X^\infty}\sum_{1\le i <j\le N}c(x_i,x_j)dQ^{\otimes\infty}\right)d\nu(Q)\\
 &=&\int_{\mathcal P(X)} \left(\int_{X^2}c(x,y)d Q\otimes Q\right)d\nu(Q).
\end{eqnarray*}
Thus the minimization \eqref{min2} is identified with the following
\[
 \inf\left\{\int_{X^2}c(x,y)d\mu_2(x,y): \mu_2=\int_{\mathcal P(X)}Q\otimes Qd\nu, \mu =\int_{\mathcal P(X)}Qd\nu,\nu\in\mathcal P(\mathcal P(X))\right\}.
\]
The existence of minimizers follows from lower semi-continuity of $c$ by the classical methods (as e.g. \cite{villani} chapter 4). Point 1 of the thesis follows if we prove that $\nu$ equal to a Dirac mass on $\mu$ is a minimizer, and that it is the unique one if $K$ is strictly convex on $\mathcal P(X)\cap\{K<\infty\}$.\\

We may assume without loss of generality that the above infimum is finite thus $K(Q)<\infty$ for $\nu$-almost every $Q$. Note that the expression
\[
 Q,Q'\in\mathcal P(X)\mapsto E(Q,Q'):=\int_{X^2}c(x,y) dQ(x) dQ'(y)
\]
is a symmetric bilinear form on $\mathcal M(X)$, is positive on $\mathcal P(X)$  and satisfies $K(Q)=E(Q,Q)$. Then
\begin{equation}\label{diff}
 \int_{X^2}c(x,y)d\mu_2(x,y) - \int_{X^2}c(x,y)d\mu(x)d\mu(y) 
\end{equation}
can be re-expressed in terms of $E,K,\nu$ as follows. In the hope to achieve better clarity, and for the next lines only, we use the notation $\int_Af(x)d\lambda(x):=\langle f,\lambda\rangle_A$ for integrals on $A=X$ or $A=X^2$. First, just by definition and since $c$ is measurable,
\[
  \int_{X^2}c(x,y)d\mu_2(x,y)=\left\langle c,\int_{\mathcal P(X)}Q\otimes Q d\nu(Q)\right\rangle_{X^2}:= \int_{\mathcal P(X)}\left\langle c,Q\otimes Q\right\rangle_{X^2} d\nu(Q).
\]
Any lower semi-continuous function is pointwise supremum of a sequence of continuous bounded functions. If $c_n:X\times X\to [0,\infty[$ are continuous bounded symmetric functions such that pointwise $c_n\uparrow c$, then the functions $(y,Q)\mapsto\langle c_n(\cdot,y), Q \rangle_{X}$ are continuous with respect to the topology of weak-* convergence on $\mathcal P(X)$. Therefore their pointwise supremum (obtained by monotone convergence against fixed $Q$) is $(y,Q)\mapsto\langle c(\cdot,y), Q \rangle_{X}$ and therefore this function is still $\mathcal B\otimes\mathcal B^*$-measurable. We may thus use Tonelli's theorem and the finiteness hypothesis \eqref{mufinite} to write
\begin{eqnarray*}
 \langle c, \mu\otimes\mu\rangle_{X^2}&:=&\int_{\mathcal P(X)}\left\langle \int_{\mathcal P(X)} \langle c(\cdot,y), Q \rangle_{X}\ d\nu(Q)\ , Q'\right\rangle_{X}d\nu(Q')\\
 &=&\int_{\mathcal P(X)^2} \langle c, Q\otimes Q'\rangle_{X^2}\ d\nu(Q)d\nu(Q').
\end{eqnarray*}
In particular $E(Q,Q')<\infty$ for $\nu\otimes\nu$-almost every $(Q,Q')$. Back in our previous notation, we may now rewrite \eqref{diff} as
\begin{equation}\label{eformula}
 \int_{\mathcal P(X)^2}(E(Q,Q) - E(Q,Q'))d\nu(Q)d\nu(Q').
\end{equation}
Using the polarization formula (valid again by monotone convergence arguments)
\[
 2E(Q,Q') + K(Q) + K(Q')=K(Q+Q'),
\]
the quadraticity of $K$ and the symmetry of the formula \eqref{eformula} in $Q,Q'$ we have 
\[
 \text{\eqref{eformula} }=\frac{1}{2}\left(\int_{\mathcal P(X)}K(Q)d\nu(Q) -\int_{\mathcal P(X)^2}K\left(\frac{Q+Q'}{2}\right)d\nu(Q)d\nu(Q')\right).
\]
To justify the above and at the same time finish the proof, note that by convexity of $K$ and since $\nu$ is a probability measure
\begin{eqnarray*}
 \int_{\mathcal P(X)^2} K\left(\frac{Q+Q'}{2}\right)d\nu(Q)d\nu(Q') &\le& \int_{\mathcal P(X)^2}\frac{1}{2}\left(K(Q)+K(Q')\right)d\nu(Q)d\nu(Q')\\& =& \int_{\mathcal P(X)}K(Q)d\nu(Q),
\end{eqnarray*}
and if $K$ is strictly convex on $\mathcal P(X)\cap\{K<\infty\}$ then equality holds precisely for $\nu$ a Dirac mass.\\

If for all $Q,Q'\in\mathcal P(X)$ we denote $\nu:=\frac{1}{2}(\delta_Q + \delta_{Q'}),\mu:= \frac{1}{2}(Q + Q')$, and we test against $\nu$ the fact that $\gamma_0=\mu^{\otimes\infty}$ is a minimizer (resp. unique minimizer) of our problem, then from the same computations as above we obtain that $K$ is convex (resp. strictly convex) on $\mathcal P(X)\cap\{K<\infty\}$, whence the thesis.
\end{proof}

\begin{rmk}[extension to the ``locally compact separable'' case]\label{extend}
As mentioned in the introduction, the proof can be adapted to the case where (a) $X$ is required only to be a locally compact separable space (b) the Borel $\sigma$-algebra on $X$ is replaced by the Baire $\sigma$-algebra in the definition of $\mathcal P(X)$ and (c) the $\sigma$-algebra $\mathcal B^*$ on $\mathcal P(X)$ is also defined to be the Baire $\sigma$-algebra relative to the weak-* topology on $\mathcal P(X)$.\\

Note (\cite{boga} chapter 6) that the Baire $\sigma$-algebra is the smallest one with respect to which all continuous functions are measurable; it is the same as the Borel one in Polish spaces but might be smaller than it if $X$ is only assumed to be locally compact separable rather than Polish (that is, if $X$ is not second-countable).\\

Theorem \ref{decorthm} extends to the current more general setting (see \cite{hs} Thm. 7.4) and the semi-continuity of $c$ still yields the existence of a minimizer as above; our justification of Tonelli's theorem, which only uses the duality with continuous functions and lower semi-continuity of $c$, works the same way since for testing measurability we use the Baire $\sigma$-algebras only.
\end{rmk}

\section*{Acknowledgement}
I wish to warmly thank Codina Cotar and Gero Friesecke for sharing their insights on their work with me. I learned most of the background material for this work during the 2014 Fields Thematic Semester on Variational Problems in Physics, Economics and Geometry. I thank the Fields Institute, without whose relaxed and interactive atmosphere this work would probably not exist. This research is funded by a scholarship of the FSMP.

\bigskip

\noindent Mircea Petrache:
\newline {\tt mircea.petrache@upmc.fr}
\end{document}